\documentclass[11pt]{amsart}
\usepackage{amsmath,amsfonts,amssymb,amsthm,amsbsy}
\usepackage{fancybox}
\usepackage{enumerate}

\newtheorem{theorem}{Theorem}[section]
\newtheorem{lemma}[theorem]{Lemma}
\newtheorem{proposition}[theorem]{Proposition}

\theoremstyle{definition}

\theoremstyle{remark}
\newtheorem{remark}[theorem]{Remark}

\numberwithin{equation}{section}

\title[Bivariate NEFs with quadratic diagonal of VF]{Bivariate natural exponential families with quadratic diagonal of the variance function}
\author[J.Matysiak]{Joanna Matysiak}
\address{Wydzia{\l} Matematyki i Nauk Informacyjnych\\
Politechnika Warszawska\\
Koszykowa 75, 00-662 Warsaw, Poland}
\email{J.Matysiak@mini.pw.edu.pl}
\date{}

\begin{document}
\begin{abstract}
We characterize bivariate natural exponential families having the diagonal of the variance function of the form
\[
\textrm{diag} V(m_1,m_2)=\left(Am_1^2+am_1+bm_2+e,Am_2^2+cm_1+dm_2+f\right),
\]
with $A<0$ and $a,\ldots,f\in\mathbb{R}$. The solution of the problem relies on finding the conditions under which a specific parametric family of functions consists of Laplace transforms of some probability measures.
\end{abstract}
\maketitle

\section{Introduction}\label{s:intro}

 Let us recall some basic concepts concerning natural exponential families (NEFs for short). For a positive measure $\mu$ on $\mathbb{R}^n$ we define its Laplace transform by $$L_{\mu}(\theta)=\int_{\mathbb{R}^n}\exp\left<\theta,x\right>\mu({dx}).$$ Let $\Theta(\mu)=\textrm{Int} \{\theta \in \mathbb{R}^n :L_{\mu}(\theta)<\infty\}$ and let $k_{\mu}=\log L_{\mu}$ denote the cumulant function of $\mu$.
Set $\mathcal{M}$ as the set of such measures $\mu$ that $\Theta(\mu) \ne \emptyset$  and $\mu$ is not concentrated on an affine hyperplane of $\mathbb{R}^n$. For $\mu \in \mathcal{M}$, the family of probabilities
 \[
F(\mu)=\{P(\theta,\mu)(dx)=\exp\left(<\theta,x>-k_{\mu}(\theta)\right)\mu(dx):\theta \in \Theta\mathbb({\mu)}\}
\]
is called \emph{the natural exponential family generated by} $\mu$, see \cite{bib:3}. The image $M_{F}$ of $\Theta(\mu)$ by diffeomorphism $k'_{\mu}$ is called \emph{the domain of means of} $\mu$. \emph{The variance function} (VF) of the NEF is defined by $V_{F}(\mathbf{m})=k''_{\mu}(\Psi_{\mu}(\mathbf{m}))$, where $\mathbf{m} \in M_{F}$ and $\Psi_{\mu}=(k_\mu^\prime)^{-1}$. In what follows we will be frequently omitting the subscript $F$ in the notation $V_F$ for the variance function.

The principal significance of the mapping $\mathbf{m} \mapsto V_{F}(\mathbf{m})$  is that, together with its domain $M_{F}$, it characterizes NEF uniquely. It makes it possible to describe NEFs assuming a concrete form of the variance function. There has been a lot of interest in such questions. For example, Letac \cite{bib:14} characterizes all NEFs  with $V_{F}(\mathbf{m})=B\mathbf{m}+C$, where $B$ is a linear operator mapping $\mathbb{R}^n$ into $\mathbb{S}_{n}$ (the space of $n\times n$ symmetric real matrices), and $C \in \mathbb{S}_{n}$. Casalis \cite{bib:7}, \cite{bib:8} gives a generalization of this result by considering $V_{F}(\mathbf{m})=a\mathbf{m}\otimes\mathbf{m}+B\mathbf{m}+C$.
Her result was further extended by Hassairi and Zarai \cite{bib:15} to NEFs with cubic variance functions.

For another standpoint, we recall some papers that provide characterizations based on a fragmentary knowledge of the variance function. Kokendji and Seshadri \cite{bib:11} start with $\det V(\mathbf{m})=const$ and thus identify the Gaussian law in $\mathbb{R}^{n}$. In \cite{bib:2} Letac and Weso{\l}owski characterized NEFs with $V(\mathbf{m})=p^{-1}\mathbf{m}\otimes \mathbf{m}-\phi(\mathbf{m})M_{\nu}$, where $M_{\nu}$ is a symmetric matrix associated with a quadratic form $\nu$, $\mathbf{m} \rightarrow \phi(\mathbf{m})$ an unknown real function, and $p$ is a number.

One of the most important papers that base on a partial knowledge of $V_{F}$ is \cite{bib:1}, in which the authors considered the diagonal family of NEFs in $\mathbb{R}^{n}$ such that $$\textrm{diag} V(\mathbf{m})=(f_1(m_1),\ldots,f_n(m_n))$$ ($\textrm{diag}V$ stands for the diagonal elements of the variance matrix V, and $f_i$ is an arbitrary function of $i$-th coordinate of $\mathbf{m}$). They gave the full characterization of the family (showing also that  $f_{i}$, $i=1,2, \ldots, n$ must be some polynomials of degree at most 2).

A natural question is the following: let us fix $n$ classes of functions $\mathcal{F}_{i}$ of the variables $m=(m_1,m_2,\ldots,m_{n})$ for $i=1,\ldots,n$. For instance $\mathcal{F}_i$ could be the functions of the form $f_i(m)=Am_i^2+B(m)+c$, where $B$ is a linear form and $A,c \in \mathbb{R}$.
What happens if we stay focused on the diagonal $\left(f_1(m),\ldots,f_{n}(m)\right)$ of the VF such that $f_i \in \mathcal{F}_{i}$? We are going to study this question for $n=2$.
For a partial result in this direction, see \cite{bib:16}.

The paper is organized as follows. Its main results, Theorem \ref{glow} and Theorem \ref{main}, are formulated in Section \ref{s:main}. Sections \ref{s:proofmain} and \ref{s:proofglow} provide proofs of the theorems, along with some auxiliary facts. The straightforward, but lengthy proof of one of the facts, Proposition \ref{lemat}, is presented in Appendix A.

\section{Main Results}\label{s:main}
Our task in this paper is to characterize NEFs with the VF of the form
\begin{equation} \label{diag}
\textrm{diag} V(m_1,m_2)=\left(Am_1^2+am_1+bm_2+e,Am_2^2+cm_1+dm_2+f\right),
\end{equation}
with $A<0$ and $a,\ldots,f\in\mathbb{R}$. To avoid considering the VFs of the diagonal NEFs analyzed by Bar-Lev et al. in \cite{bib:1}, we need to exclude the case $b=c=0$. Due to the symmetry between $b$ and $c$, in the sequel we shall assume without loss of generality that $b\ne0$.

We decided to stick here to $A<0$ only and to place our findings on the case $A>0$ in a separate future paper. Our decision is motivated by a rather surprising lack of symmetry between the two cases. It turns out that if $A>0$ then \eqref{diag} implies a much wider and less homogeneous class of parametric bivariate functions that are candidates for the Laplace transforms of the corresponding NEFs, than for $A<0$.

\begin{remark}
An interesting feature of the problem  is that it can be formulated equivalently using regression properties. Assume that $\mathbf{X}=(X_1,X_2)$ and $\mathbf{Y}=(Y_1,Y_2)$ are  independent identically distributed random vectors  such that their Laplace transform exists in a set with a nonempty interior.  The regression formulation of the problem we solve in this paper is to identify the distributions of $\mathbf{X}$ if
\begin{eqnarray*}
\mathbb{E}\left((X_1-Y_1)^2-2AX_1Y_1|\mathbf{X}+\mathbf{Y}\right) &=&a(X_1+Y_1)+b(X_2+Y_2)+2e,\\
\mathbb{E}\left((X_2-Y_2)^2-2AX_2Y_2|\mathbf{X}+\mathbf{Y}\right) &=&c(X_1+Y_1)+d(X_2+Y_2)+2f.
\end{eqnarray*}
For other works in a similar vein, see e.g.\cite{bib:1}, \cite{bib:4}, \cite{bib:5}, \cite{bib:13}, \cite{bib:2}.
\end{remark}

Now we state the main theorem, that specifies the form of a Laplace transform corresponding to the measure generating NEF with the VF given in \eqref{diag}.
\begin{theorem}\label{glow}
Let $\mu$ be a probability measure (not concentrated in a point) generating NEF with VF given by \eqref{diag}.
Then the Laplace transform of $\mu$ is of the form
\begin{equation}\label{ttww}
L(\theta_1,\theta_2)=\left(\sum_{i=1}^{n_r}\alpha_i\exp\left(\lambda_i\theta_1+\frac{\lambda_i^2-a\lambda_i-eA}{b}\theta_2\right)\right)^{-\frac{1}{A}},\;\;
(\theta_1,\theta_2) \in \Theta(\mu),
\end{equation}
where $\{\lambda_1,\ldots,\lambda_{n_r}\}$ is the set of distinct real roots of
\begin{multline}\label{lameq}
\lambda^4-2a\lambda^3+\left(2Ae+a^2-db\right)\lambda^2-\left(2Aae-adb+cb^2\right)\lambda+\\A^2e^2-edbA+fb^2A=0
\end{multline}
(the coefficients in \eqref{lameq} come from \eqref{diag}). Furthermore, $n_r \geq 2 $ and
either
\begin{itemize}
\item  $\alpha_i \geq 0$, $i=1,\ldots,n_r$, $\sum_{i=1}^{n_r}\alpha_i=1$, $-1/A \in \mathbb{N}$ and $\Theta(\mu) = \mathbb{R}^2$, or
\item  $\alpha_i \leq 0$, $i=1,\dots,n_r$, $\sum_{i=1}^{n_r}\alpha_i=-1$, $-1/A \in 2\mathbb{N}$ and $\Theta(\mu)=\mathbb{R}^2$.
\end{itemize}
\end{theorem}

\begin{remark}
It is clear that if one of the bullet item conditions above holds then \eqref{ttww} is the Laplace transform of a probability measure.
\end{remark}

It might be worthy to emphasize that the only assumption we make about $\Theta(\mu)$ is that it contains a neighborhood of the origin (this follows from the fact that $\mu$ belongs to NEF).
The concrete form of $\Theta(\mu)$ ($\Theta(\mu)=\mathbb{R}^2$) is not assumed but it is implied by the assumptions of Theorem \ref{glow}.

The proof of Theorem \ref{glow} contains  two parts. The first part explains why the distribution sought in \eqref{diag}, must have the Laplace transform of the form \eqref{ttww}. We present this partial result in Proposition \ref{lemat}.
The second part consists of  some  arguments concerning the conditions on $A$, $\alpha_i's$ and $\Theta(\mu)$.
 The arguments are gathered in a stronger result (see Theorem \ref{main} below), which, by a special choice of a matrix $\Lambda$ in the case of $n_r \geq 3$, implies the conditions on $\alpha_i$ and $A$ (the case $n_r=2$ will be considered  separately in the proof of Theorem \ref{glow}).

The abovementioned stronger result is the following
\begin{theorem}\label{main}
Let $r>0$, $\Lambda=(\Lambda_1,\Lambda_2,\Lambda_3)^{T} \in \mathbb{R}^{3 \times 3}$ and suppose $\Theta \subset \mathbb{R}^3$ contains a neighborhood of the origin. Define
\begin{equation}\label{lfunc}
L(\theta)=\left(\alpha_0+\sum_{i=1}^3\alpha_i\exp\left<\Lambda_i,\theta\right>\right)^{r},\;\; \theta\in\Theta.
\end{equation}
Assume that equation
\begin{equation} \label{star}
 \mathbf{a}^{T}\Lambda=\boldsymbol{0}
\end{equation}
has no solutions in $a=(a_1,a_2,a_3) \in \mathbb{Z}^{3}$ such that $\exists_{i \ne j} \; a_ia_j<0$.
Then $L$ is Laplace transform of a probability measure  if and only if either
\begin{enumerate}[(a)]
\item  $\alpha_i \ge 0$, $i=0,\ldots,3$, $\sum_{i=0}^3\alpha_i=1$ and  $r \in \mathbb{N}$, or
\item $\alpha_i \leq 0$, $i=0,\ldots,3$, $\sum_{i=0}^3\alpha_i=-1$ and $r \in 2\mathbb{N}$.
\end{enumerate}
\end{theorem}
Notice that conditions (a) and (b) are sufficient for $L$ to be the Laplace transform of a probability measure.
\begin{remark}\label{barlevremark}
Bar-Lev et al. [Proposition 3.1(a),\cite{bib:1}] consider a similar class of Laplace transforms.
If $n\in \mathbb{N}$, define $\mathcal{T}$ to be the family of non-empty subsets of $\{1,2,\ldots,n\}$, and for $\mathbf{z}=(z_1,\ldots,z_n)$ put
\[
\mathbf{z}^{T}=\prod_{j\in T}z_j,\;\;for\;\; T \in \mathcal{T}.
\]
Proposition 3.1 from \cite{bib:1} describes for which $\alpha$ the analytic function of $\mathbf{z}$
\begin{equation}\label{barlev}
(1+\sum_{T \in \mathcal{T}}\alpha_{T}\mathbf{z}^{T})^N, \;\;N\in\mathbb{N},
\end{equation}
has non-negative coefficients in the Taylor expansion in the neighborhood of the origin. For $z_j=e^{\theta_j}$, $j=1,\ldots,n$, this is equivalent to deciding whether \eqref{barlev} is the Laplace transform of a positive measure.

\noindent Observe that our approach differs from the  one taken in \cite{bib:1}. This is because  in our problem \eqref{lfunc}, the crucial step  is to justify the series expansion of $L$, while in \eqref{barlev} analyticity is assumed.  The  $\Lambda$ matrix and the support structure condition \eqref{star} allow as to identify the masses of atoms among the series coefficients and to infer about the signs of $\alpha$'s.
As a consequence of the need to justify the series expansion, we restrict ourselves to $\mathbb{R}^3$. Another difference lies in the fact that the support of the measure corresponding to \eqref{barlev}  belongs to a lattice in $\mathbb{N}^d$. The support of the measure corresponding to \eqref{lfunc} depends on the choice of $\Lambda$.
\end{remark}



\begin{remark}
Again, it might be worthy to emphasize that our assumptions on $\Theta$ in Theorem \ref{main} are weak. They are dictated by the fact that the measure with Laplace transform \eqref{lfunc} belongs to a NEF.
\end{remark}
\section{Proof of Theorem \ref{main}}\label{s:proofmain}

\begin{proof}[Proof of Theorem \ref{main}]
\noindent
 Let us consider first  $L$ on $\mathcal{D}$:
\begin{equation*}
\mathcal{D}=\left\{\theta \in \mathbb{R}^{3}:\sum_{i=0}^{3}\alpha_i\exp\left<\Lambda_i,\theta\right>>0\right\}.
\end{equation*}
We shall be using the following relation between two vectors $x=(x_1,x_2,x_3)$ and $y=(y_1,y_2,y_3)$ in $\mathbb{R}^3$:
\[ x \prec y \Longleftrightarrow x_1<y_1 \vee (x_1=y_1 \wedge x_2<y_2) \vee (x_1=y_1 \wedge x_2=y_2 \wedge x_3<y_3).
\]
Observe that \eqref{star} implies the existence of permutation  $\sigma$ of $\{0,1,2,3\}$ such that
\begin{equation}\label{porz}
 \Lambda_{\sigma_0} \prec \Lambda_{\sigma_1}\prec \Lambda_{\sigma_2} \prec \Lambda_{\sigma_3},
\end{equation}
where $\Lambda_0=(0,0,0)$. We will split the reasonings into four cases.
\begin{enumerate}
\item Assume first that $\alpha_{\sigma_0}>0$. We will show that $\alpha_{\sigma_i} \geq 0$, $i=1,2,3$ and $r \in \mathbb{N}$.
 Note that \eqref{porz} yields
\begin{equation}\label{prec0}
\Lambda_0 \prec \Lambda_{\sigma_1}-\Lambda_{\sigma_0} \prec \Lambda_{\sigma_2}-\Lambda_{\sigma_0}\prec \Lambda_{\sigma_3}-\Lambda_{\sigma_0}.
\end{equation}
This in turn  allows us to choose $\theta^{\star} \in \mathcal{D}$ (negative with arbitrarily large modulus) to ensure that
\begin{equation*}
\left|\sum_{i=1}^{3}\frac{\alpha_{\sigma_i}}{\alpha_{\sigma_0}}\exp\left<\Lambda_{\sigma_i}-\Lambda_{\sigma_0},\theta^{\star}\right>\right|<1.
\end{equation*}
Define $H(\theta)=\exp\left<-r\Lambda_{\sigma_0},\theta\right>\frac{L(\theta+\theta^{\star})}{L(\theta^{\star})}$. ($H$ is the Laplace transform of a probability measure if and only if $L$ is the Laplace transform of a probability measure.).  Function $H$ has the following series expansion in the neighborhood of the origin:
\[
H(\theta)=C\sum_{j=0}^{\infty}\frac{r(r-1)\cdots(r-j+1)}{j!}\left(\sum_{i=1}^3\frac{\alpha_{\sigma_i}}{\alpha_{\sigma_0}}\exp\left<\Lambda_{\sigma_i}-\Lambda_{\sigma_0},\theta^{\star}+\theta\right >\right)^{j},
\]
where $C=\big(\sum_{i=0}^{3}\frac{\alpha_{\sigma_i}}{\alpha_{\sigma_0}}\exp\left<\Lambda_{\sigma_i}-\Lambda_{\sigma_0},\theta^{\star}\right>\big)^{-r}$.

\noindent From \eqref{prec0} we see that    $Cr\frac{\alpha_{\sigma_1}}{\alpha_{\sigma_0}}\exp\left<\Lambda_{\sigma_1}-\Lambda_{\sigma_0},\theta^{\star}\right>$ is the only coefficient at $\exp\left<\Lambda_{\sigma_1}-\Lambda_{\sigma_0},\theta\right>$  in the expansion of $H$, hence $\alpha_{\sigma_1} \geq 0$.
Since \eqref{star}  assures that there is no $a \in \mathbb{N}$,  such that $\Lambda_{\sigma_2}-\Lambda_{\sigma_0}=a\left(\Lambda_{\sigma_1}-\Lambda_{\sigma_0}\right)$, we get that $Cr\frac{\alpha_{\sigma_2}}{\alpha_{\sigma_0}}\exp\left<\Lambda_{\sigma_2}-\Lambda_{\sigma_0},\theta^{\star}\right>$ is the only coefficient at $\exp\left<\Lambda_{\sigma_2}-\Lambda_{\sigma_0},\theta\right>$ in the expansion of $H$, so $\alpha_{\sigma_2} \geq 0$. Analogously, from \eqref{star} there are no $a_1$, $a_2 \in \mathbb{N}$ such that
$\Lambda_{\sigma_3}-\Lambda_{\sigma_0}=a_1\left(\Lambda_{\sigma_1}-\Lambda_{\sigma_0}\right)+a_2\left(\Lambda_{\sigma_2}-\Lambda_{\sigma_0}\right)$, hence $Cr\frac{\alpha_{\sigma_3}}{\alpha_{\sigma_0}}\exp\left<\Lambda_{\sigma_3}-\Lambda_{\sigma_0},\theta^{\star}\right>$ is the only coefficient at $\exp\left<\Lambda_{\sigma_3}-\Lambda_{\sigma_0},\theta\right>$ in the expansion of $H$, and $\alpha_{\sigma_3} \geq 0$.

Now suppose that $r \notin \mathbb{N}$. Consequently, there exists the smallest $k$ such that $r(r-1)\cdots(r-k+1)<0$.
Since coefficient at $\exp\left<k\left(\Lambda_{\sigma_1}-\Lambda_{\sigma_0}\right),\theta\right>$ is positive as a mass of an atom in $k\left(\Lambda_{\sigma_1}-\Lambda_{\sigma_0}\right)$ there exist $n_1,n_2,n_3 \in \mathbb{N}$ such that
$k\left(\Lambda_{\sigma_1}-\Lambda_{\sigma_0}\right)=n_1\left(\Lambda_{\sigma_1}-\Lambda_{\sigma_0}\right)+ n_2\left(\Lambda_{\sigma_2}-\Lambda_{\sigma_0}\right)+ n_3\left(\Lambda_{\sigma_3}-\Lambda_{\sigma_0}\right)$  which contradicts \eqref{star}; hence $r\in\mathbb{N}$.

\item Now assume that $\alpha_{\sigma_0} \leq 0$ and $\alpha_{\sigma_3} >0$. We will show that in fact $\alpha_{\sigma_0}=0$,  $\alpha_{\sigma_i} \geq 0$, $i=1,2$ and $r \in \mathbb{N}$.

 Note that \eqref{porz} yields
 \begin{equation}\label{prec3}
 \Lambda_{\sigma_0}-\Lambda_{\sigma_3} \prec \Lambda_{\sigma_1}-\Lambda_{\sigma_3}\prec \Lambda_{\sigma_2}-\Lambda_{\sigma_3}\prec\Lambda_0.
 \end{equation}
Therefore  we can choose $\theta^{\star}\in\mathcal{D}$  satisfying
\[
\left|\sum_{j=0}^{2}\frac{\alpha_{\sigma_j}}{\alpha_{\sigma_3}}\exp\left<\Lambda_{\sigma_j}-\Lambda_{\sigma_3},\theta^{\star}\right> \right|<1
\]
and  define $G(\theta)=\exp\left<-r\Lambda_3,\theta\right>\frac{L(\theta+\theta^{\star})}{L(\theta^{\star})}$. ($G$ is the Laplace transform of a probability measure if and only if $L$ is the Laplace transform of a probability measure.) We can write its series expansion:

\[
G(\theta)=C\sum_{j=0}^{\infty}\frac{r(r-1)\cdots(r-j+1)}{j!}
\bigg(\sum_{i=0}^{2}\frac{\alpha_{\sigma_i}}{\alpha_{\sigma_3}}\exp\left<\Lambda_{\sigma_i}-\Lambda_{\sigma_3},\theta + \theta^{\star}\right> \bigg)^j,
\]
where $C=\left(\sum_{i=0}^{3}\frac{\alpha_{\sigma_i}}{\alpha_{\sigma_3}}\exp\left<\Lambda_{\sigma_i}-\Lambda_{\sigma_3},\theta^{\star}\right>\right)^{-r}$.

\noindent From \eqref{prec3} we see that $Cr\alpha_{\sigma_3}^{-1}\alpha_{\sigma_2}\exp\left<\Lambda_{\sigma_2}-\Lambda_{\sigma_3},\theta^{\star}\right>$ is the only coefficient at $\exp\left<\Lambda_{\sigma_2}-\Lambda_{\sigma_3},\theta\right>$ in the expansion of $G$, hence $\alpha_{\sigma_2}\geq 0$. Using the same reasoning as in the preceding case, we conclude from \eqref{star} that $Cr\alpha_{\sigma_3}^{-1}\alpha_{\sigma_i}\exp\left<\Lambda_{\sigma_i}-\Lambda_{\sigma_3},\theta^{\star}\right>$ is the only coefficient at $\exp\left<\Lambda_{\sigma_i}-\Lambda_{\sigma_3},\theta\right>$ in the expansion of $G$, hence $\alpha_{\sigma_i} \geq 0$, $i=0,1$, so $\alpha_{\sigma_0}=0$.

In order to conclude that $r\in\mathbb{N}$ it suffices to repeat the previous reasoning.
\item Now assume that $\alpha_{\sigma_0} \leq 0$ and $\alpha_{\sigma_3}<0$ (or $\alpha_{\sigma_0} < 0$ and $\alpha_{\sigma_3}\leq 0$). We will show that this assumption leads to contradiction.

 Since $\alpha_{\sigma_1}+\alpha_{\sigma_2} = 1-\alpha_{\sigma_0}-\alpha_{\sigma_3}$, we conclude that $\alpha_{\sigma_1}>\frac{1}{2}$ or $\alpha_{\sigma_2}>\frac{1}{2}$. Without  loss of generality we let $\alpha_{\sigma_1} >\frac{1}{2}$.

Again \eqref{porz} implies
\begin{equation*}
 \Lambda_{\sigma_0}-\Lambda_{\sigma_1} \prec\Lambda_0\prec \Lambda_{\sigma_2}-\Lambda_{\sigma_1}\prec \Lambda_{\sigma_3}-\Lambda_{\sigma_1}.
\end{equation*}

 Set $T(\theta)=\exp\left<-r\Lambda_{\sigma_1},\theta\right>L(\theta)$ and expand it in a neighborhood of the origin:
\[
T(\theta)=\alpha_{\sigma_1}^{r}\sum_{k=0}^{\infty}\frac{r(r-1)\cdots(r-k+1)}{k!\alpha_{\sigma_1}^k} \bigg( \sum_{i=0,i\neq 1}^3 \alpha_{\sigma_i}\exp\left<\Lambda_{\sigma_i}-\Lambda_{\sigma_1},\theta\right>\bigg)^{k}.
\]

\noindent From \eqref{star}, $r\alpha_{\sigma_i}\alpha_{\sigma_1}^{r-1}$ is the only coefficient at $\exp\left<\Lambda_{\sigma_i}-\Lambda_{\sigma_1},\theta\right>$ in the expansion of $T$, hence $\alpha_{\sigma_i} \geq 0$ for $i=0,2,3$, so $\alpha_{\sigma_0}=\alpha_{\sigma_3}=0$. It contradicts the  assumption.

\item Now assume that $\alpha_{\sigma_0}=\alpha_{\sigma_3}=0$. A  simplification of the analysis given in cases (1) and (2) leads to $\alpha_{\sigma_1} \geq 0$ and $\alpha_{\sigma_2} \geq 0$. Once we are done with it, the fact that $r\in\mathbb{N}$ follows from the same steps as in the first case.
\end{enumerate}
 Summarizing, we have  shown that $L$ (defined on $D$) is a Laplace transform of a probability measure if and only if (a) holds.

 Assume now that $r \in 2\mathbb{N}$. If so, we can consider $L$
 \begin{equation*}
L(\theta)=\left(\alpha_0+\sum_{i=1}^3\alpha_i\exp\left<\Lambda_i,\theta\right>\right)^{r}
=\left(-\alpha_0-\sum_{i=1}^3\alpha_i\exp\left<\Lambda_i,\theta\right>\right)^{r},
 \end{equation*}
 on
 \begin{multline*}
 \mathcal{D}^{'}=\left\{\theta \in \mathbb{R}^{3}:\sum_{i=0}^{3}\alpha_i\exp\left<\Lambda_i,\theta\right><0 \right\} =\\ \left\{\theta \in \mathbb{R}^{3}:\sum_{i=0}^{3}-\alpha_i\exp\left<\Lambda_i,\theta\right>>0\right\} .
 \end{multline*}

 Defining $\tilde{\alpha_i}=-\alpha_i$, $i=0,\ldots,3$, from (a) we can conclude that $\tilde{\alpha_i} \geq 0$, hence $\alpha_i \leq 0$, $i=0,\ldots,3$. Thus we arrive at (b).
\end{proof}

\section{Proof of Theorem \ref{glow}} \label{s:proofglow}
\subsection{From the form of the diagonal \eqref{diag} to the characteristic equation \eqref{lameq}}
Here we shall justify the transition from   the diagonal \eqref{diag} to the characteristic equation \eqref{lameq}.

Let $k$ be a  cumulant function of  a measure $\mu$  belonging to NEF satisfying \eqref{diag} (that is $k(\theta_1,\theta_2)=\log L(\theta_1,\theta_2)$, $(\theta_1,\theta_2)\in\Theta(\mu)$). Then condition \eqref{diag}  can be written equivalently as
\begin{eqnarray}
\frac{\partial^{2}{k}}{\partial \theta_1^2} &=&A\left(\frac{\partial k}{\partial \theta_1}\right)^2+a\frac{\partial k}{\partial \theta_1} + b \frac{\partial k}{\partial \theta_2} + e, \label{cum1}\\
\frac{\partial^{2}{k}}{\partial \theta_2^2} &=&A\left(\frac{\partial k}{\partial \theta_1}\right)^2+c\frac{\partial k}{\partial \theta_1} + d\frac{\partial k}{\partial \theta_2} + f\label{cum2}.
 \end{eqnarray}
Define $R=e^{-Ak}$. Then  \eqref{cum1} and \eqref{cum2} become
\begin{eqnarray}
\frac{\partial^2 R}{\partial \theta_1^2} & = & a\frac{\partial R}{\partial \theta_1} + b\frac{\partial R}{\partial \theta_2} - eAR, \label{r1}\\
\frac{\partial^2 R}{\partial \theta_2^2} & = & c\frac{\partial R}{\partial \theta_1} + d\frac{\partial R}{\partial \theta_2} - fAR.\label{r2}
 \end{eqnarray}

Since we assume that $b\ne 0$, as a consequence of \eqref{r1} and \eqref{r2} we get
\begin{multline}\label{rtet1}
\frac{\partial^4 R}{\partial \theta_1^4}-2a\frac{\partial^3R}{\partial\theta_1^3}+\left(2Ae+a^2-db\right)\frac{\partial^2R}{\partial\theta_1^2} -\left(2aeA-adb+cb^2\right)\frac{\partial R}{\partial \theta_1}\\ + A^2e^2-edbA+fb^2A=0,
\end{multline}
with the characteristic equation \eqref{lameq}.
\begin{remark}
If instead of  making the assumption on $b$, one assumes  that $c \ne 0$, then, due to symmetry between $b$ and $c$, one obtains
\begin{multline}\label{rtet2}
\frac{\partial^4 R}{\partial \theta_2^4}-2d\frac{\partial^3R}{\partial\theta_2^3}+\left(2Af-ac+d^2\right)\frac{\partial^2R}{\partial\theta_2^2} -\left(2dfA-acd+bc^2\right)\frac{\partial R}{\partial \theta_2}\\ + A^2f^2-acfA+ec^2A=0,
\end{multline}
 which is analogous to  \eqref{rtet1},  with the characteristic polynomial
\begin{equation*}
\nu^4-2d\nu^3+\left(2Af-ac+d^2\right)\nu^2-\nu\left(bc^2-acd+2dfA\right)+A^2f^2-acfA+ec^2A=0.
\end{equation*}
The roots  of this equation in $\nu$ and the roots of \eqref{lameq} are connected via
\begin{eqnarray}
\lambda^2&=&a\lambda+b\nu-eA, \label{par1}\\
\nu^2&=&c\lambda+d\nu-fA.\label{par2}
\end{eqnarray}
\end{remark}

Careful analysis of the roots of \eqref{lameq} leads, via various solutions of  \eqref{rtet1}, to the following

\begin{proposition}\label{lemat}
Let $\mu$ be a probability measure (not concentrated in a point) generating NEF with VF given by \eqref{diag}.
Let $\{\lambda_1,\ldots,\lambda_{n_r}\}$ be the set of distinct real roots of \eqref{lameq}. Then  $n_r \geq 2$  and the Laplace transform of $\mu$ is
\begin{equation}
L(\theta_1,\theta_2)=\left(\sum_{i=1}^{n_r}\alpha_i\exp\left(\lambda_i\theta_1+\frac{\lambda_i^2-a\lambda_i-eA}{b}\theta_2\right)\right)^{-\frac{1}{A}},
\end{equation}
with $(\theta_1,\theta_2) \in \Theta(\mu)$.
\end{proposition}
The proof will be presented in Appendix \ref{dowod_prop}.
\begin{remark}
Analyzing  the form of the Laplace transform $L$ given in Proposition 4.2, one can see that the measure corresponding to the case $A=-1$ consists of point masses which are concentrated on a parabola (see \eqref{par1}).
\end{remark}

Now, we are in a position to prove our main result.
\begin{proof}[Proof of Theorem \ref{glow}]
Let $(X,Y)$ be a random vector with Laplace transform \eqref{ttww}. If $$(\tilde{X},\tilde{Y})=(X-\lambda_1/A,bY+aX+eA-\lambda_1^2/A)$$ then the  Laplace transform of $(\tilde{X},\tilde{Y})$ is
\begin{equation}\label{tildel}
\tilde{L}(\theta_1,\theta_2)=\left(\alpha_1+\sum_{i=2}^{n_r}\alpha_i\exp\left[(\lambda_i-\lambda_1)\theta_1+(\lambda_i^2-\lambda_1^2)\theta_2\right]\right)^{-\frac{1}{A}},
\end{equation}
where $(\theta_1,\theta_2) \in \Theta(\tilde{\mu})$ and $\tilde{\mu}=\mathcal{L}(\tilde{X},\tilde{Y})$. We shall be working with
the Laplace transform \eqref{tildel} rather than with \eqref{ttww} because the form of \eqref{tildel} allows us to use Theorem \ref{main}.

Now we will split our reasonings with respect to the number of distinct roots of \eqref{lameq}.

First, if $n_r=3$ or $4$, we will indicate how to specify $\Lambda$ in Theorem \ref{main} to make it answer the questions about the coefficients from \eqref{tildel} (as a consequence also for \eqref{ttww}).
\begin{itemize}
\item For $n_r=4$ we choose
 \begin{equation}\label{lambda1}
 \Lambda=\left[\begin{array}{ccc}
 \lambda_2-\lambda_1 & \lambda_2^2-\lambda_1^2 & 0\\
 \lambda_3-\lambda_1 & \lambda_3^2 -\lambda_1^2& 0\\
 \lambda_4 -\lambda_1& \lambda_4^2-\lambda_1^2 & 0
 \end{array}\right].
 \end{equation}
 Obviously this particular $\Lambda$ plugged into \eqref{lfunc} yields \eqref{tildel} (after changing $\alpha_0,\alpha_1,\alpha_2,\alpha_3$ into $\alpha_1,\alpha_2,\alpha_3,\alpha_4$, respectively). What we need to show is that such $\Lambda$ satisfies \eqref{star}.
 To this end we first assume without loss of the generality that $\lambda_2<\lambda_3<\lambda_4$.  Equation \eqref{star} can be rewritten equivalently as the system of three equations
 \begin{equation*}
 a_1(\lambda_i-\lambda_1)+a_2(\lambda_i^2-\lambda_1^2)=0,\quad i=2,3,4.
 \end{equation*}
From the first two  we get
 \begin{equation*}
a_1+a_2(\lambda_i+\lambda_1)=0,\quad i=2,3.
 \end{equation*}
 Subtracting the first ($i=1$) from the second one ($i=2$) yields
\begin{equation*}
a_2(\lambda_2-\lambda_1)=0,\label{ap21}
\end{equation*}
hence $a_2=0$ and $a_1=0$. Therefore \eqref{lambda1} satisfies \eqref{star}.
\item For $n_r=3$ we take
 \begin{equation}\label{lambda2}
 \Lambda=\left[\begin{array}{ccc}
 \lambda_2-\lambda_1 & \lambda_2^2-\lambda_1^2 & 0\\
 \lambda_3-\lambda_1 & \lambda_3^2 -\lambda_1^2& 0\\
 0& 0 & 0
 \end{array}\right].
 \end{equation}
Analogously to the preceding case, after plugging \eqref{lambda2} into \eqref{lfunc} we obtain \eqref{tildel} and the same analysis shows that $\eqref{lambda2}$ satisfies \eqref{star}.
\end{itemize}

For $n_r=2$, \eqref{tildel} becomes
\begin{equation*}
\tilde{L}(\theta_1,\theta_2)=\left(\alpha_1+\alpha_2\exp((\lambda_2-\lambda_1)\theta_1+(\lambda_2^2-\lambda_1^2)\theta_2)\right)^{-\frac{1}{A}}.
\end{equation*}
In order to prove the necessary condition, it is enough to analyze one-dimensional Laplace transform $l$ given by
$$l(\theta)=\tilde{L}(\theta/(\lambda_2-\lambda_1),0)=\left(\alpha_1+\alpha_2\exp(\theta)\right)^{-\frac{1}{A}}$$ in a neighborhood of $0$.
Our aim is to show that the necessary  conditions for $l$ to be a Laplace transform of a probability measure on $\mathbb{R}$, are
either
\begin{enumerate}
\item $\alpha_1,\alpha_2 \geq 0$, $\alpha_1+\alpha_2=1$ and $-\frac{1}{A} \in \mathbb{N}$, or
\item $\alpha_1,\alpha_2 \leq 0$, $\alpha_1+\alpha_2=-1$ and $-\frac{1}{A} \in 2\mathbb{N}$
\end{enumerate}
(sufficiency of the conditions is clear).

Let us consider $l$ on $\mathcal{D}_l=\{\theta: \alpha_1+\alpha_2\exp(\theta) >0\}$, what implies $\alpha_1+\alpha_2=1$ (recall that $l(0)=1$). Therefore, without  loss of generality one can assume that $\alpha_1 \geq 1/2$ (so $\alpha_2 \leq 1/2)$. What we want to prove is that $\alpha_1 \geq 1/2$ implies $\alpha_2 \geq 0$ . If $\alpha_1=1/2$, then $\alpha_2=1-\alpha_1=1/2$, so let us assume that $\alpha_1 > 1/2$. In such case $|{\alpha_2}/{\alpha_1}|=|{(1-\alpha_1)}/{\alpha_1}| <1$, hence, in a neighborhood of the origin, $l$ can be written as
\begin{equation}\label{elmale}
l(\theta)=\alpha_1^{-1/A}\sum_{k=0}^{\infty}\frac{-1/A(-1/A-1)...(-1/A-k+1)}{k!}\left(\frac{\alpha_2}{\alpha_1}\right)^k\exp(k\theta).
\end{equation}
The only coefficient at $\exp(\theta)$ in \eqref{elmale} is $\alpha_2\alpha_1^{-1/A-1}$, hence  $\alpha_2\alpha_1^{-1/A-1} \geq 0$ and so $\alpha_2 \geq 0$.

Now we focus on the part of conclusion dealing with the exponent $r=-{1}/{A}$.  We  provide here a reasoning analogous to the one given for $r$ in the proof of Theorem \ref{main}. Suppose that $r \notin \mathbb{N}$. Then there exists the smallest integer $k$ such that $r(r-1)...(r-k+1) <0$, so the coefficient at $\exp(k\theta)$ in \eqref{elmale} is negative. This contradicts the fact that $l$ is a Laplace transform of a discrete probability measure (with non-negative point masses). Thus $r \in \mathbb{N}$. Concluding, we get (1).

In order to get (2), let us now consider $l$ with $r=-{1}/{A} \in 2\mathbb{N}$ on $D^{'}_{l}=\{\theta:\alpha_1+\alpha_2\exp(\theta) <0\}$ (condition $l(\theta) >0$ is satisfied on $\theta \in D_{l}^{'}$). Since $r$ is even,
\begin{equation*}
l(\theta)=\left(\alpha_1+\alpha_2\exp(\theta)\right)^{r}=\left(-\alpha_1-\alpha_2\exp(\theta)\right)^{r},
\end{equation*}
and we analyze it on $D_{l}^{'}=\left\{\theta:-\alpha_1-\alpha_2\exp(\theta) >0\right\}$. Denoting $\tilde{\alpha_1}=-\alpha_1$ and $\tilde{\alpha_2}=-\alpha_2$,
we arrive at the case considered in (1). Therefore  $\tilde{\alpha_1}$, $\tilde{\alpha_2} \geq 0$, which yields $\alpha_1, \alpha_2 \leq 0$. Thus (2) follows.
 \end{proof}

\subsection*{Acknowledgement}
I would like to thank  Jacek Weso\l owski for encouragement and many helpful discussions.
\begin{appendix}
\section{}
\subsection{Proof of Proposition \ref{lemat}}\label{dowod_prop}
In order to prove  Proposition \ref{lemat} we will analyze the roots of \eqref{lameq} and examine corresponding $R$ functions. Our aim will be to eliminate the solutions of \eqref{rtet1} that do not lead (via $L=R^{-1/A}$) to Laplace transforms of  probability measures. To do so, we will treat  $R$ as a function of $\theta_1$ only (with fixed $\theta_2$), and we shall use auxiliary Lemma \ref{prop1} and Lemma \ref{prop2} presented below to reject some inadmissible solutions.

\begin{lemma}\label{prop1}
Let $r>0$. Assume that $P_m$ is a polynomial of degree $m$ over $\mathbb{R}$ and define
\begin{equation}\label{eqpropo1}
f(\theta)=P_{m}(\theta)+\sum_{i=1}^{k}A_{i}\exp(\lambda_i\theta)+B\theta\exp(\gamma \theta),\;\;\;  \theta \in \Theta,
\end{equation}
where $\Theta$ contains some neighborhood of zero,  $\lambda_1<\lambda_2<\ldots<\lambda_k$ (and none of them is zero), and $\gamma$, $B$, $A_i$ are some real numbers. \\
If $f^r$ is a Laplace transform of a probability measure then  $P_m \equiv P_0 $ and $B=0$.
\end{lemma}
\begin{proof}
Since $f$ is defined in a neighborhood of zero, we consider a characteristic function $\phi(t)=f^r(it)$. We have
\begin{displaymath}
|\phi(t)|^2=\left|P_{m}(it)+Bit\exp(\gamma it)+\sum_{i=1}^{k}A_i\exp(\lambda_i it)\right|^{2r}.
\end{displaymath}
Function $|\phi|$ is bounded on $\mathbb{R}$ as the absolute value of a characteristic function of a probability measure. Therefore $P_{m}\equiv P_{0}$ and $B=0$.
\end{proof}
\begin{lemma}\label{prop2}
Let $r>0$ and
\begin{multline*}
f(\theta)=\sum_{j=0}^{1}e^{\lambda_j\theta}\left[(A_{0j}+\theta B_{0j})\cos(\gamma_j\theta)+  (A_{1j}+\theta B_{1j})\sin(\gamma_j\theta)\right]\\+\sum_{j=2}^{3}A_{j}\exp(\lambda_j\theta),\;\;\;\theta\in\Theta,
\end{multline*}
where $\Theta$ contains some neighborhood of zero, be a function with all the parameters being some real numbers. Furthermore, assume that $\lambda_{0}+i\gamma_0 \ne \lambda_1+i\gamma_1$. Then if
$f^r$ is a Laplace transform of a probability measure, then  $A_{0j}=A_{1j}=B_{0j}=B_{1j}=0$, $j=0,1$.
\end{lemma}

\begin{proof}
Let $\phi(t)=f^r(it)$ be a characteristic function corresponding to
  $f^r$, then
\begin{multline*}
|\phi(t)|^{2}=\\\bigg|\sum_{j=0}^{1}\left(e^{\lambda_jit}(A_{0j}+itB
_{0j})\frac{e^{\gamma_{j}t}+e^{-\gamma_jt}}{2} +  e^{\lambda_jit}(A_{1j}+itB_{1j})\frac{e^{-\gamma_{j}t}-e^{\gamma_jt}}{2i}\right) \\
+A_{2}\exp(i\lambda_2t)+A_3\exp(i\lambda_3t)\bigg|^{2r}.
\end{multline*}
Since $|\phi|$ is bounded on $\mathbb{R}$, we arrive at the conclusion.
\end{proof}
\begin{proof}[Proof of Proposition \ref{lemat}]
We shall separately consider all situations regarding the roots of \eqref{lameq}.
\subsubsection{Four single real roots of \eqref{lameq}}
Let $\lambda_1,\lambda_2,\lambda_3,\lambda_4$ be the roots of \eqref{lameq}. Then $R$ is of the form
\begin{equation}\label{fdrr}
R(\theta_1,\theta_2)=\sum_{i=1}^4A_i(\theta_2)\exp(\lambda_i\theta_1),
\end{equation}
where $A_{i}(\cdot)$, $i=0,\ldots,3$ are some real functions. Plugging  \eqref{fdrr} into \eqref{r2} we obtain
\begin{equation*}
\lambda_i^2A_{i}(\theta_2)-a\lambda_iA_i(\theta_2)-bA_{i}^{'}(\theta_2)-eAA_i(\theta_2)=0,\;\; i=1,2,3,4.
\end{equation*}
Therefore the explicit formulas for $A_{i}(\cdot)$'s are
\begin{equation*}
A_i(\theta_2)=A_i\exp\left(\frac{\lambda_i^2-a\lambda_i-eA}{b}\theta_2\right),\;i=1,2,3,4,
\end{equation*}
and $A_i$, $i=1,2,3,4$ are some real constants. Hence
\begin{equation}
R(\theta_1,\theta_2)=\sum_{i=1}^4A_i\exp\left(\lambda_i\theta_1+
\frac{\lambda_i^2-a\lambda_i-eA}{b}\theta_2\right).
\end{equation}
\subsubsection{One double and two single real roots of \eqref{lameq}}\label{odts}
Let  $\lambda_1, \lambda_2$ be single roots and $\lambda_3$ a double root of \eqref{lameq}. Then $R$ is of the form
\begin{equation}\label{drtd}
R(\theta_1,\theta_2)=\sum_{i=1}^2A_i(\theta_2)\exp(\lambda_i\theta_1)+\left(A_3(\theta_2)+A_4(\theta_2)\theta_1\right)\exp(\lambda_3\theta_1),
\end{equation}
where $A_i(\cdot)$, $i=1,2,3,4$, are  some real functions. Since $R^{-1/A}$ is a Laplace transform of a probability measure, using Lemma \ref{prop1} we conclude that $A_4 \equiv 0$. Plugging \eqref{drtd} into \eqref{r1} we obtain
\begin{equation*}
\lambda_i^2A_i(\theta_2)-a\lambda_iA_i(\theta)-bA_i^{'}(\theta_2)-eAA_i(\theta_2) =0,\;i=1,2,3.
\end{equation*}
These yield
\begin{eqnarray*}
A_i(\theta_2)=A_i\exp\left(\frac{\lambda_i^2-a\lambda_i-eA}{b}\theta_2\right),\;i=1,2,3,
\end{eqnarray*}
and
\begin{equation*}
R(\theta_1,\theta_2)=\sum_{i=1}^{3}A_i\exp\left(\lambda_i\theta_1+
\frac{\lambda_i^2-a\lambda_i-eA}{b}\theta_2\right).
\end{equation*}
\subsubsection{One single and one triple real roots of \eqref{lameq}}
Let $\lambda_1$ be a single and $\lambda_2$ a triple real root of \eqref{lameq}. Then
\begin{equation*}
R(\theta_1,\theta_2)=A_1(\theta_2)\exp\left(\lambda_1\theta_1\right)+ \left(A_2(\theta_2)+A_{3}(\theta_2)\theta_1+A_{4}(\theta_2)\theta_1^2\right)\exp(\lambda_2\theta_1),
\end{equation*}
where $A_i(\cdot)$, $i=1,2,3,4$, are  some real functions. Analogously as in the previous subsection, using Lemma \ref{prop1} (for $R(\theta_1,\theta_2)\exp(-\lambda_2\theta_1)$), we conclude that
$A_3\equiv A_4 \equiv 0$. Function $R$ in such case is of the form
\begin{equation*}
R(\theta_1,\theta_2)=\sum_{i=1}^{2}A_i\exp\left(\lambda_i\theta_1+\frac{\lambda_i^2-a\lambda_i-eA}{b}\theta_2\right),
\end{equation*}
where $A_i$, $i=1,2$ are some real functions.
\subsubsection{Quadruple real root  of \eqref{lameq}}
Let $\lambda_1$ be a quadruple real root of \eqref{lameq}. Then
\begin{equation*}
R(\theta_1,\theta_2)=A_1\exp\left(\lambda_1\theta_1+\frac{\lambda_1^2-a\lambda_1-eA}{b}\theta_2\right),
\end{equation*}
where $A_1$ is a real constant.
\subsubsection{Two double real roots of \eqref{lameq}}
Let $\lambda_1$ and $\lambda_2$ be two  distinct double roots of \eqref{lameq}. Steps analogous to the ones taken in Section \ref{odts}, with the help of Lemma \ref{prop1}, yield
\begin{equation*}
R(\theta_1,\theta_2)=\sum_{i=1}^2A_i\exp\left(\lambda_i\theta_1+\frac{\lambda_i^2-a\lambda_i-eA}{b}\theta_2\right),
\end{equation*}
where $A_i$, $i=1,2$, are some real constants.
\subsubsection{Two single real and two single complex roots of \eqref{lameq}}
Let $\lambda_1$, $\lambda_2$ be two real  and $\lambda_3 +i\gamma_3$, $\lambda_3-i\gamma_3$ be two complex roots of \eqref{lameq}.
Then $R$ takes the form
\begin{multline*}
R(\theta_1,\theta_2)=\sum_{i=1}^2A_i(\theta_2)\exp\left(\lambda_i\theta_1+\frac{\lambda_i^2-a\lambda_i-eA}{b}\theta_2\right) + \\
\exp(\lambda_3\theta_1)\left(A_{3}(\theta_2)\cos(\gamma_3\theta_1)+A_4(\theta_2)\sin(\gamma_3\theta_1)\right).
\end{multline*}
From Lemma \ref{prop2} we conclude that $A_3 \equiv A_4 \equiv 0$, hence
\begin{equation*}
R(\theta_1,\theta_2)=\sum_{i=1}^{2}A_{i}\exp\left(\lambda_i\theta_1+\frac{\lambda_i^2-a\lambda_i-eA}{b}\theta_2\right),
\end{equation*}
where $A_i$, $i=2,3$ are some real constants.
\subsubsection{Four distinct complex roots of \eqref{lameq}}
Let $\lambda_1+i\gamma_1$, $\lambda_1-i\gamma_1$, $\lambda_2+i\gamma_2$ and $\lambda_2-i\gamma_2$ be the roots of \eqref{lameq}. Then $R$ takes the form
\begin{multline*}
R(\theta_1,\theta_2)= \exp(\lambda_1\theta_1)\left(A_{1}(\theta_2)\cos(\gamma_1\theta_1)+A_{2}(\theta_2)\sin(\gamma_1\theta_1)\right) +\\
\exp(\lambda_2\theta_1)\left(A_3(\theta_2)\cos(\gamma_2\theta_1)+A_4(\theta_2)\sin(\gamma_2\theta_1)\right).
\end{multline*}
Using Lemma \ref{prop2} we conclude that there does not exist probability measure  with the Laplace transform $L=R^{-1/A}$.
\subsubsection{Two double complex roots of \eqref{lameq}}
Analogously to the preceding case, there is no probability measure corresponding to $R$
\begin{multline*}
R(\theta_1,\theta_2)=\exp(\lambda_1\theta_1)\bigg[\cos(\gamma_1\theta_1)\big(A_{1}(\theta_2)+\theta_1A_{2}(\theta_2)\big) +  \\ \sin(\gamma_1\theta_1)\big(A_{3}(\theta_2)+\theta_1A_{4}(\theta_2)\big) \bigg],
\end{multline*}
where $\lambda_1+i\gamma_1$ and $\lambda_1-i\gamma_1$ are double complex roots of \eqref{lameq} and $A_{i}$, $i=1,2,3,4$, are some real functions.
\subsubsection{Double real root and two single complex roots of \eqref{lameq}}
Let $\lambda_1$, $\lambda_2+i\gamma_2$ and $\lambda_2-i\gamma_2$ be the roots of \eqref{lameq}, then
\begin{multline*}
R(\theta_1,\theta_2)=\exp(\lambda_1\theta_1)\left(A_{1}(\theta_2)+\theta_1A_{2}(\theta_2)\right) \\+ \exp(\lambda_2\theta_1)\left(A_{3}(\theta_2)\cos(\gamma_2\theta_1)+A_4(\theta_2)\sin(\gamma_2\theta_1) \right).
\end{multline*}
From Lemma \ref{prop1} and Lemma \ref{prop2} we obtain $A_{2}\equiv A_3 \equiv A_4 \equiv 0$.  Hence $R$ takes the form
\begin{equation*}
R(\theta_1,\theta_2)=\exp\left(\lambda_1\theta_1+\frac{\lambda_1^2-a\lambda_1-eA}{b}\theta_2\right).
\end{equation*}
The conclusion of Proposition \ref{lemat} follows by a straightforward aggregation of the  above points.
\end{proof}

\end{appendix}

\end{document}